\documentclass{amsart}

\usepackage[T1]{fontenc}
\usepackage[utf8]{inputenc}
\usepackage{amsmath,amsthm,amsfonts,amscd,amssymb,eucal,latexsym,mathrsfs}
\usepackage{stmaryrd}
\usepackage{enumerate}
\usepackage{hyperref}
\usepackage[all]{xy}
\usepackage{etoolbox}

\usepackage{tikz,tikz-cd}
\usetikzlibrary{shapes.geometric}
\usetikzlibrary{arrows}
\usepackage{tqft}

\usetikzlibrary{positioning}
\usepackage{float}
\usepackage{MnSymbol}
\usetikzlibrary{matrix}

\theoremstyle{plain}
\newtheorem{theorem}{Theorem}[section]
\newtheorem*{theorem*}{Theorem}

\newtheorem{corollary}[theorem]{Corollary}
\newtheorem*{corollary*}{Corollary}

\newtheorem{lemma}[theorem]{Lemma}

\newtheorem{proposition}[theorem]{Proposition}

\theoremstyle{definition}
\newtheorem{remark}[theorem]{Remark}

\newtheorem{definition}[theorem]{Definition}
\newtheorem*{definition*}{Definition}

\usetikzlibrary{calc,graphs}
\usepackage{xcolor}
\usetikzlibrary{arrows,decorations.pathmorphing}

\newcommand{\e}{\varepsilon}

\newcommand{\IR}{\mathbb{R}}

\newcommand{\ZI}{\mathbb{Z}}
\newcommand{\IN}{\mathbb{N}}

\newcommand{\SI}{\mathbb{S}}

\DeclareMathOperator{\del}{\partial}

\DeclareMathOperator{\Aut}{\mathrm{Aut}}

\DeclareMathOperator{\inj}{\hookrightarrow}

\DeclareMathOperator{\surj}{\twoheadrightarrow}

\DeclareMathOperator{\spn}{\mathrm{span}}

\newcommand{\Bord}{\mathrm{Bord}}

\DeclareMathOperator{\SAut}{\mathrm{SAut}}

\newcommand{\ts}{\textsection}

\newcommand{\ip}[1]{\langle#1\rangle} 

\newcommand{\dirlim}{\varinjlim}
\newcommand{\St}{\mathop{\mathrm{St}}}

\title{Surgery on $\mathbf{Aut}(F_2)$}

\author{Sylvain Barr\'e}
\author{Mika\"el Pichot}
\address{Sylvain Barr\'e, UMR 6205, LMBA, Université de Bretagne-Sud,BP 573, 56017, Vannes, France}\email{Sylvain.Barre@univ-ubs.fr}
\address{Mika\"el Pichot, McGill University, 805 Sherbrooke St W., Montr\'eal, QC H3A 0B9, Canada}\email{pichot@math.mcgill.ca}

\begin{document}
\begin{abstract}
We study a geometric construction of certain finite index subgroups of $\Aut(F_2)$.
\end{abstract} 

\maketitle

We recall that $\Aut(F_2)$ admits an isometric properly discontinuous action with compact quotient on a CAT(0) complex $X_0$, called the Brady complex, which was introduced in \cite{brady1994automatic}

In \ts \ref{S - 1}, we show that $\Aut(F_2)$ can be presented (virtually) in a very simple manner from a labelling of a flat torus.  
Starting from a torus of size $6\times n$, for some fixed integer $n\geq 1$ (we shall discuss the case $n=5$ in  details), we associate to it, via a ``pinching and (systolic) filling'' construction,  a 2-complex $B_n$, with fundamental group a group $G_n$ which is of finite index  in $\Aut(F_2)$.  
The universal cover $X_n$ of $B_n$ is  a CAT(0) space.  
  We show in \ts \ref{S - 1} that the $X_n$'s are pairwise isometric for every $n\geq 1$, and that  $X_0$ and $X_n$ are locally isometric, in the sense that their vertex links are pairwise isometric (Lemma \ref{L - link}) for every $n\geq 0$. This  implies, by the result below,  that  $X_n$ is isometric to $X_0$ for every $n$.

In \ts 2, we prove a geometric rigidity theorem for the Brady complex.   Roughly speaking, the result states that $X_0$ is the ``free complex'' on one (any) of its face, among the complexes locally isomorphic to $X_0$ (see Th.\ \ref{T - rigid} for a precise statement).  This seems to be a rather special property of $X_0$, which is not very often satisfied among the 2-complexes  we have studied. 

Theorem \ref{T - rigid} implies that every CAT(0) 2-complex locally isometric to $X_0$ is isometric to $X_0$.
The notion of local isomorphism in this statement is slightly more restrictive than requiring the existence of an abstract isometry between the links shown in Lemma \ref{L - link}: the two complexes must be of the same (local) type (see  \ts\ref{S - 2}). The additional conditions are however immediate to verify  for the $X_n$'s for $n\geq 1$. 

In \ts\ref{S - 1.5}, we show that every torsion free finite index orientable subgroup of $\Aut(F_2)$ can be constructed abstractly by a pinching--and--filling construction, similar to the one given in \ts\ref{S - 1}, applied to finitely many tori. It is not clear  however how to extend the explicit procedure given in \ts\ref{S - 1} to describe, e.g., the family of  torsion free finite index subgroups which are associated with a fixed number of tori.
 
In \ts\ref{S - 3}, we explain the origin of the toric presentation given in \ts\ref{S - 1}. 
The present paper can be seen as a continuation of an earlier work \cite{surg}, in which we introduce a cobordism category $\Bord_A$ which can be used  to construct groups acting on complexes of a given (local) type $A$. We show below that the  techniques of \cite{surg} can be applied to the case of $\Aut(F_2)$. It gives rise groups acting on complexes of type $\Aut(F_2)$ as defined in \ts\ref{S - 2}. 

In the case of $\Aut(F_2)$, however, the spaces  constructed by surgery in this way must, by the results in \ts\ref{S - 2}, be quotients of the Brady complex $X_0$, and the resulting fundamental groups, subgroups of $\Aut(F_2)$. This is not true of many cobordism categories, and contrasts for example with the categories studied in \cite{surg}, in which the groups accessible by surgery in a given category (of a fixed local type, e.g., Moebius--Kantor) are typically not  pairwise commensurable. Again, the category $\Bord_A$ is  rather special in this respect when $A$ is the type $\Aut(F_2)$.

Finally,  we give in \ts\ref{S - 4} an  example of a CAT(0) 2-complex $X'$ which is locally isomorphic to the Brady complex to $X_0$, but not isometrically isomorphic to it. Here ``locally isomorphic'' refers to  the fact that the links in  $X'$ are isometric to the links in the complex $X_0$.

\bigskip

\textbf{Acknowledgement.} The second author is supported by an NSERC discovery grant. 

\section{The toric presentation}\label{S - 1}

Consider the flat torus $T_5$ of size $6\times 5$ defined as follows:

\begin{figure}[H]
\includegraphics[width=13cm]{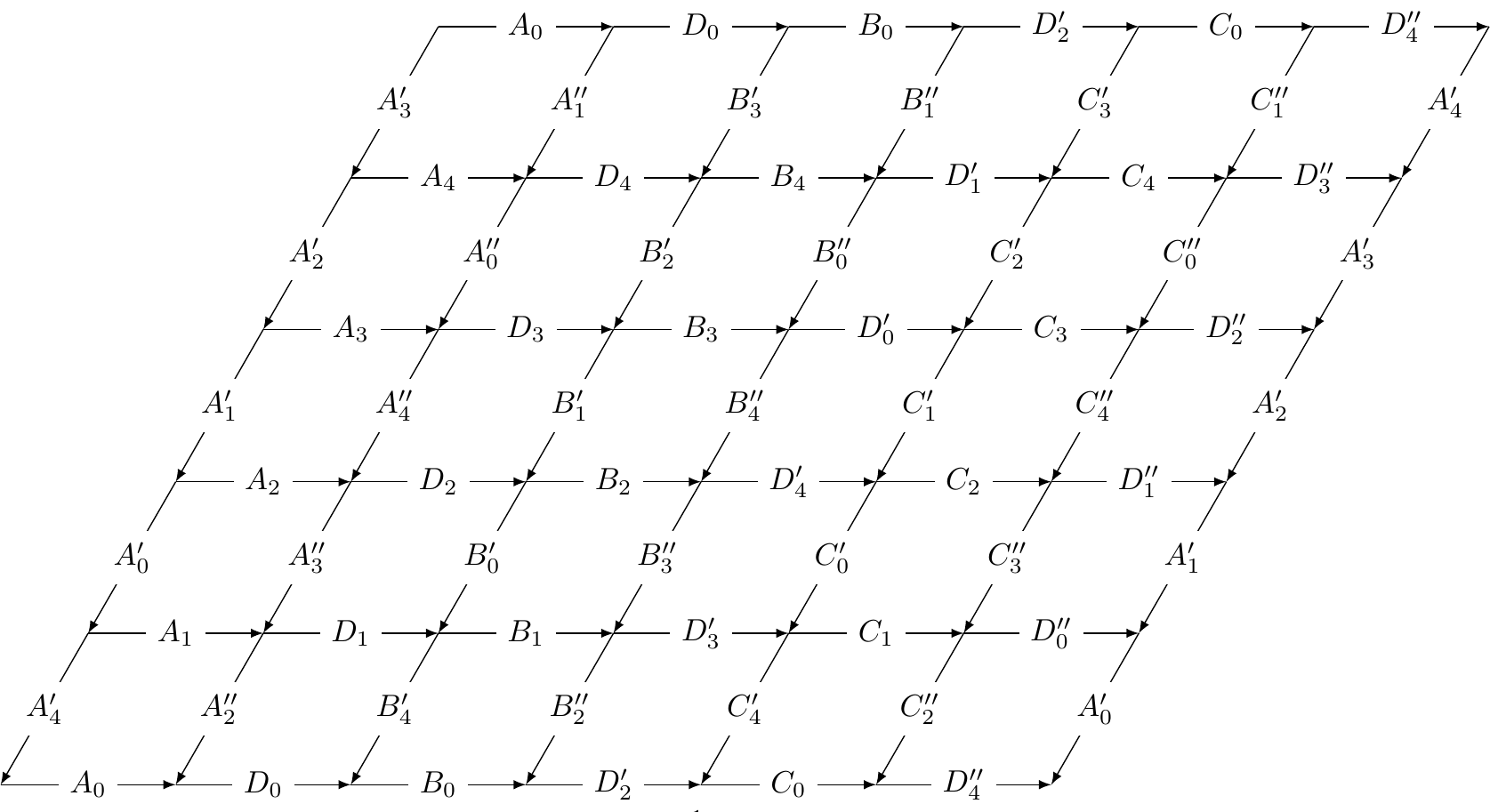}
\end{figure}

\noindent Every edge in $T_5$ is oriented and labelled. The boundary is identified in the standard way respecting both  the orientation and the labelling of the boundary edges.  

Note that there is a non trivial Dehn twist, that we will denote $\tau_{-6}$, 
in the vertical direction. 

We endow the torus $T_5$ with the standard Euclidean metric, in which the cells are (as shown in the figure) lozenges with sides of length 1.

Here is the basic construction. 

The figure contains a total of 20 letters. They are denoted $A_r$, $B_r$, $C_r$, $D_r$, $0\leq r\leq 4$. Let $L$ be a letter. For every triple $K$ of the form 
\[
K:=(L,L',L'')
\] 
 consider an oriented  triangle with edges labelled by $K$ in the given order. We attach this triangle to the torus $T_5$ along its boundary, respecting the orientation and labelling for the boundary edges. 
This operation, repeated for the twenty triples $K$, defines a 2-complex $B_5$.

Let $X_5:=\widetilde{B_5}$ denote the universal cover of $B_5$,  and $G_5:=\pi_1(B_5)$ denote its fundamental group of $B_5$.

Note that the canonical map $T_5\to B_5=X_5/G_5$ is not injective on vertices.  One may view $B_5$ as a ``wrinkled presentation'' of the group $G_5$ and the map $T_5\to B_5$ as the ``sewing map''. Observe furthermore that every triple $K$  ``jumps'' on the torus $T_5$. (We  call $K$ a ``knight''.)

By definition, a \emph{jump} on $T_5$ is an oriented edge between two vertices of $T_5$. Every triple $K$ defines three jumps, from the extremity of an edge in $K$ to the origin of the consecutive edge, modulo 3.  

\begin{lemma}
Jumps are either disjoint or they share a common support.   
\end{lemma}

\begin{proof}
A jump associated with a triple $K$ corresponds either to  the affine transformation
\[
\begin{cases} x\mapsto x+1 \mod 6\\ y\mapsto y-2\mod 5\end{cases}
\]
where $x$ is even modulo 6, or to its inverse
\[
\begin{cases} x\mapsto x-1 \mod 6\\ y\mapsto y+2\mod 5\end{cases}
\]
where $x$ is odd modulo 6. It is not difficult to show that these two transformations do not depend on $K$. Since they are inverse of each other, jumps with a  common vertex must have the same support. 
\end{proof}

In particular, the jumps define an involution $\sigma$ of the vertex set of $T_5$, whose orbit partition $T_5/\ip \sigma$ coincides with  the vertex set of $X_5/G_5$.

Let us orient the torus $T_5$ counterclockwise, and consider the positive labelling $\in \{1,2,3,4\}$ of the edges issued from a vertex, where $1$ refers to the positive real axis. The basic construction induces a permutation of the labels associated with every jump. We shall now describe this permutation.

\begin{lemma}
The permutation of $\{1,2,3,4\}$ associated with the jump 
\[
\begin{cases} x\mapsto x+1 \mod 6\\ y\mapsto y-2\mod 5\end{cases}
\]
is  the 4-cycle $(1,2,4,3)$.
\end{lemma}

This shows that the resulting permutation does not depend on $K$; the permutation associated with the opposite jump is the inverse permutation. 

\begin{proof}
Let us for example do the bottom left corner $(0,0)$, which is mapped to $(1,-2)=(1,3)$ under $\sigma$.  The corresponding transformation of the counterclockwise labelling  reads
\[
\begin{cases} 1 = A_0\\ 2=A_4'\\ 3=D_3''\\ 4=A_3'
\end{cases}\mapsto \begin{cases} 1 = D_3\\ 2=A_0''\\3=A_3\\4=A_4'' 
\end{cases}
\] 
which corresponds to label permutation $(1,2,4,3)$.
\end{proof}

The following shows that $X_5$ is locally isomorphic isomorphic to the Brady complex in the (usual) sense that their links are pairwise isomorphic. 

\begin{lemma}\label{L - link}
Every link in $X_5$ is isomorphic to the link of the Brady complex. 
\end{lemma}

\begin{proof}
We shall compute the links in $X_5$. That it is isomorphic to that of the Brady complex follows from \cite{brady1994automatic,crisp2005classification} (see also \ts \ref{S - 2} below). Since the expression for $\sigma$ is independent of the base point in $T_5$, it is enough to check the link of the origin. We represent the links at the origin and its image in $T_5$ as follows (the drawings respects the scale provided by the angle metric): 
\begin{figure}[H]
\includegraphics[width=3cm]{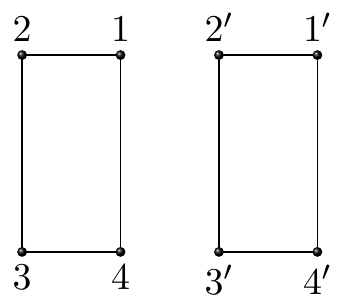}
\end{figure}
\noindent The prime labels correspond to the image $(1,-2)$. According to the previous lemma, edges in the link of $X_5$ corresponds to the permutation $s=(1,2,4,3)$. This defines four additional edges in the above figure: $(x,s(x)')$ for every $x\in \{1,2,3,4\}$. It is straightforward to check  that this graph is the link of $X_0$ (compare \ts \ref{S - 2}).
\end{proof}

The basic construction can be generalized to an arbitrary integer $n\geq 1$ in the following way. 

Suppose first that $n$ is a sufficiently large integer (e.g., $n\geq 4$). Consider a torus $T_n$ of size $6\times n$, where the vertical identification involves a Dehn twist $\tau_{-6}$. For every letter $L$ on $((x,y),(x,y-1))$, where $x$ is even, write labels $L'$ and $L''$ on, respectively, $((x,y-2),(x+1,y-2))$ and   $((x+1,y-3),(x+1,y-4))$; for every letter $L$ on $((1,y),(2,y))$, write labels $L'$ and $L''$ on, respectively, $((3,y-2),(4,y-2))$ and   $((5,y-4),(6,y-4))$. Then the same construction for every triple $K=(L,L',L'')$ on $4n$ letters defines a 2-complex $B_n$ which is locally isomorphic to the Brady complex. The notation $B_n$ is consistent with the previous notation $B_5$.

One can further extend this construction of $B_n$ to every integer $n\geq 1$ as follows. Let $T_\infty:=\dirlim T_n$ (with respect to partial embeddings from a base point) is a cylinder with an obvious action of $\ZI$. Since the set of triples (knights) is $\ZI$-invariant, this action descends to the basic construction $B_\infty$; we let, by definition, $X_n$ is the universal cover of the quotient $B_n$ of this space by $n\ZI$. The notation $B_n$ is again consistent.  Note however than the description using knights  is only clearly visible for $n$ suficiently large ($n\geq 4$ is large enough). 

This shows the following:

\begin{proposition}
The space $X_n$ are pairwise isomorphic for $n\geq 1$.
\end{proposition}  
\begin{proof}
They have a common cover $B_\infty$.
\end{proof}
 
In the next section we give a different proof of this fact, which includes isomorphism with the Brady complex $X_0$.

\section{Geometric rigidity}\label{S - 2}

We shall describe the local data by a type (or ``local type''), following  \cite[\ts 4]{surg}. In the latter paper we were interested in two sorts of types, simplicial and metric. In the present paper, we shall use \emph{labelled types}, which add connecting maps to mark the link edges using angle labels as follows (cf.\ \cite[Rem.\ 4.5]{surg}).

\begin{definition}\label{D - types}
A  \emph{labelled type} (in dimension 2) is
\begin{enumerate}
\item a set of graphs (the links);
\item a set of marked shapes, i.e., polygons with filled interior and labelled angles;
\item a set of connecting maps marking every link edge with an angle label.   
\end{enumerate}
\end{definition}

We define the  type $\Aut(F_2)$ as follows:
\begin{enumerate}
\item the link of the Brady complex; it is isomorphic to the graph (see \cite[Fig.\ 6]{crisp2005classification})
\begin{center}
\includegraphics[width=5cm]{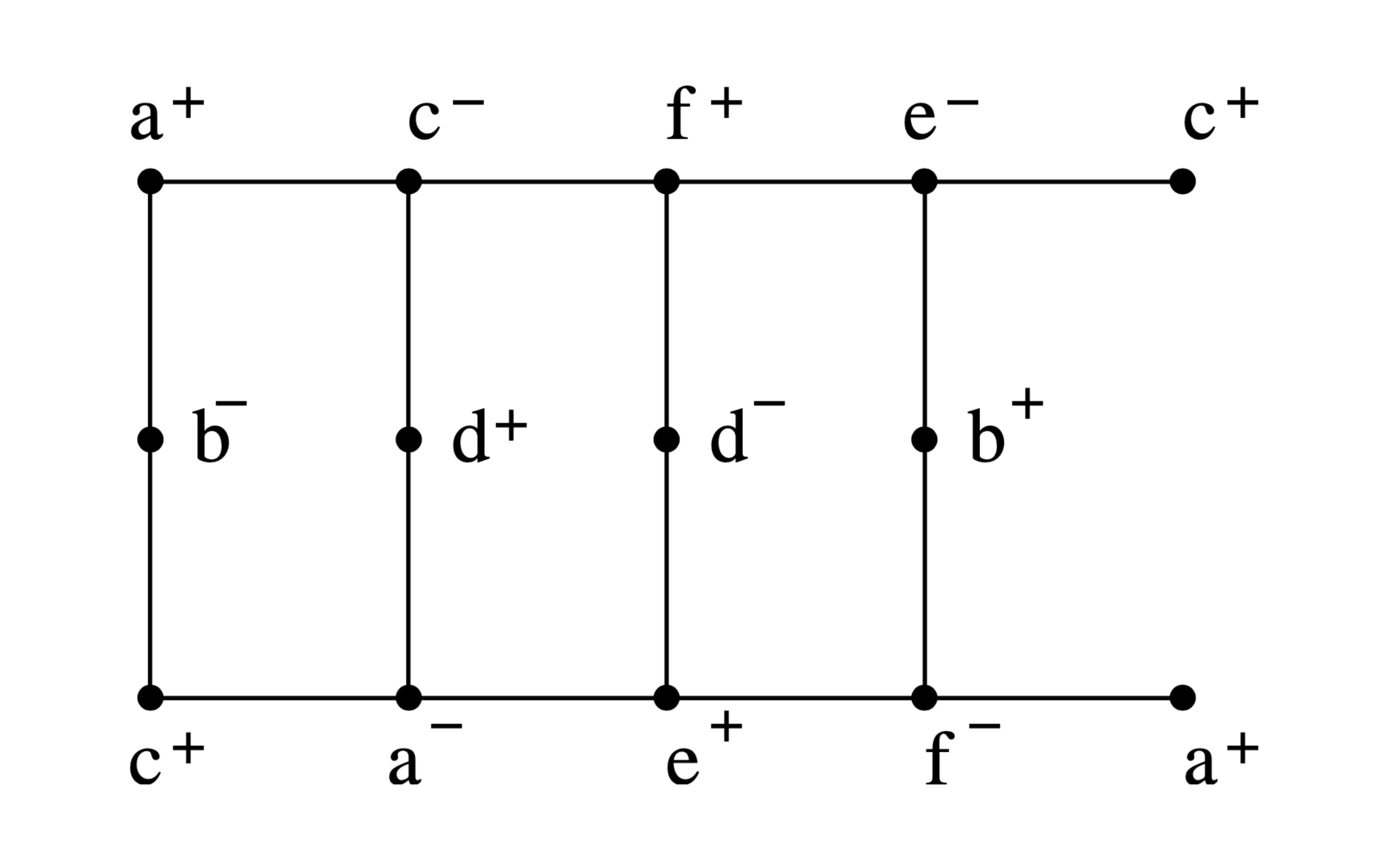}
\end{center}
The letters are associated  (see \cite[\ts 3]{crisp2005classification} for details) with the presentation 
\begin{align*}
\ip{ a,b,c,d,e,f \mid\ & ba=ae=eb,\ de=ec=cd,\\
& bc = cf = fb,\ df = fa = ad,\\
&ca = ac,\ ef = fe }
\end{align*}
\noindent of the braid group $B_4$.

\item two shapes, a lozenge and an equilateral triangle, labelled in the following way: 
\begin{center}
\includegraphics[width=5cm]{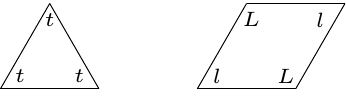}
\end{center}
\item a connecting map defined by 
\begin{center}
\includegraphics[width=4.5cm]{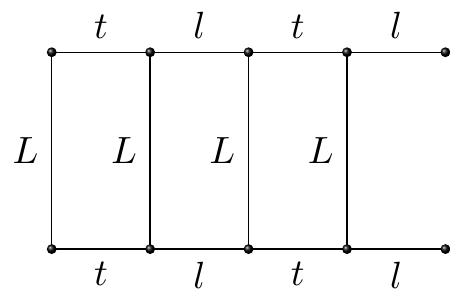}
\end{center}
\end{enumerate}

Let $T$ be a labelled type. We say that a 2-complex with labelled face angles is \emph{of type $T$} it has the correct links and shapes, and the induced marking of the link edges corresponds to a connecting map. A homomorphism between two complexes of type $T$ is a 2-complex homomorphism which preserves the angle labels.

The following is straightforward to verify from, e.g., the original description of $X_0$ in \cite{brady1994automatic}.

\begin{proposition}
The Brady complex $X_0$ is of type $\Aut(F_2)$.  
\end{proposition}

Our main theorem in this section is a converse of this statement. 
More precisely, we prove that the complex $X_0$ satisfies a \textbf{universal property}: it is freely generated by any of its faces. 

\begin{theorem}\label{T - rigid} Let $X$ be a 2-complex of type $\Aut(F_2)$.
Let $S$ be a face in $X_0$ and let $f\colon S\to X$ be a label and shape preserving map from $S$ to a face in $X$. There exists a unique homomorphism $\tilde f\colon X_0\to X$ whose restriction to $S$ coincides with $f$. Furthermore, $\tilde f$ is a covering map onto its image. 
\end{theorem}

Every 2-complex of type $\Aut(F_2)$ can be naturally endowed with a metric structure, in which the triangle face is equilateral and the lozenge a union of two equilateral triangles. By the link condition, every such a complex is locally CAT(0). Every homomorphism between complexes of type $\Aut(F_2)$ is isometric, and conversely, every isometry preserves the angle labels. The universal property in the metric situation states that if $f\colon S\to X$ is an isometry between a face $S$ of $X_0$ and a face of $X$, then there exists a unique isometry $\tilde f\colon X_0\to X$ whose restriction to $S$ coincides with $f$.

\begin{lemma}
Let $X$ be a 2-complex of type $\Aut(F_2)$.
Let $S$ be a face in $X_0$ and let $f\colon S\to X$ be a map identifying $S$ with a face in $X$. Let $p$ be a vertex of $S$. There exists a unique label preserving extension $\tilde f\colon \St_p(X_0)\to X$ of $f$ to the star of $p$ in $X_0$.  
\end{lemma}
\begin{proof} Let $L_0$ denote the link of $p$ in $X_0$ and $L$ the link of $f(p)$ in $X$. 
The map $f$. The map $f$ induces a label preserving map from an edge $e_0$ in $L_0$ to an edge $e$ in $L$. Since the labels incident to an arbitrary vertex in $L_0$ and $L$ are identical, and the labels around a vertex are pairwise distinct, there exists a unique label preserving extension of $f$ to the faces adjacent to $S$ containing $p$. More generally, it is easy to check that the map $e_0\to e$ admits a unique label preserving extension to a graph isomorphism $L_0\to L$. This shows that $f$ admits a unique label preserving extension  $\tilde f\colon \St_p(X_0)\to X$.
\end{proof}

\begin{proof}[Proof of Theorem \ref{T - rigid}]
We refer to the standard CAT(0) structure on $X_0$ defined before the lemma. Let $C$ be a maximal ball  in $X_0$ centred in $S$ to which $f$ admits an unique extension. We let $f$ denote this extension. Suppose for towards a contradiction that $C$ has a finite radius. 

Let $p\in \del C$. If $p$ belongs to the interior of a face, it is obvious how to extends $f$ to an $\e$-neighbourhood of  $p$ in $X_0$. Suppose that $p$ belongs to the interior of an edge $e$, and let $f$ be the unique face containing $e$ and intersecting the interior of $C$. Since both $X_0$ and $X$ are of type $\Aut(F_2)$, there exists a unique extension of $f$ to an $\e$-neighbourhood of  $p$ in $X_0$. Assume now that $p$ is a vertex of $\del C$. In this case, $C$ contains a face, and the previous lemma shows that $f$ can be extended in a unique way to an $\e$-neighbourhood of $p$. 

Furthermore, that if $p,p'$ are two points in $\del C$ at distance $\leq 1$, then the two extensions of $f$ from $p$ and $p'$ coincide on their intersection. Since $\del C$ is compact, this shows that $f$ can be extended to an $\e$-neighbourhood of $C$, contradicting the maximality of $C$.

Finally, $\tilde f$ is a covering map by construction.
\end{proof}

\begin{corollary}
The spaces $X_n$ are pairwise isomorphic for every $n\geq 0$. 
\end{corollary}

\begin{proof}
Since $X_n$ is of type $\Aut(F_2)$, we have a covering map $X_0\to X_n$. Since $X_n$ is simply connected, this map is an isomorphism.  
\end{proof}

\begin{corollary}
The groups $G_n$ are of finite index in $\Aut(F_2)$. 
\end{corollary}

\begin{proof}
The special automorphism group $\SAut(F_2)$, which is of index 2 in $\Aut(F_2)$, acts transitively on the set of triangles in $X_0$ by the description in  \cite{crisp2005classification}. If $f$ is a triangle in $X_0$, and $s$ an element in $G_n$, then there exists a unique element $t_s\in \SAut(F_2)$ whose restriction to $f$ coincide with $s$. By the theorem, $s$ and $T-s$ coincide on $X_0$, and the map $s\mapsto t_s$ provides an embedding of $G_n$ into $\SAut(F_2)$.   
\end{proof}
	
Another corollary, Theorem \ref{Th - frame} below, shows that the Brady complex admits a ``frame'' in the following sense. 

We recall that a flat plane in $X_0$ is an isometric embedding $\IR^2\inj X_0$ of the standard Euclidean plane in $X_0$. 

 \begin{definition}
A \emph{frame} on $X_0$ is an orientation, and a labelling by two letters $e$ and $f$, of the edge set of $X_0$, such that for every flat plane $\Pi$ in $X_0$ which is a union of lozenges, the following holds
\begin{enumerate}
\item the ordered set $B_x:=(e_x,f_x)$ of outgoing edges at a vertex $x$ in $\Pi$, with respective labels $e$ and $f$,  forms a basis of $\Pi$,  
\item  the unique translation of $\Pi$ which takes a vertex $x$ to a vertex $y$ takes the ordered set $B_x$ to the ordered set $B_y$.
\end{enumerate}
 \end{definition}
 
 Thus a frame is a way to move a basis consistently along the various embeddings $\Pi$ into $X_0$.

\begin{theorem}\label{Th - frame}
There exists a frame on $X_0$.
\end{theorem}

\begin{proof}
The map $X_0\surj B_5$ is a covering map. We consider the obvious frame on the torus $T_5$, and the induced orientation and labelling of the edge set of $B_5$ by the sewing map $T_5\to B_5$ (which is a bijection on the edge set), and lift the orientation and labelling to $X_0$ using the map $X_0\surj B_5$. Since every flat plane in $X_5$ maps onto the image of $T_5$ in $B_5$, this defines a frame on $X_0$. 
\end{proof}

We shall say that a group of automorphisms of $X_0$ is \emph{orientable} if it preserves the frame constructed in Theorem \ref{Th - frame}.

Note that every finite index subgroup of $\Aut(F_2)$ contains a finite index subgroup which is orientable. 

\begin{proof}
Let $G$ be a finite index subgroup of $\Aut(F_2)$. Then the group $G\cap G_5$ is of finite index in $\Aut(F_2)$: 
\[
[\Aut(F_2):G\cap G_5]\leq [\Aut(F_2):G][\Aut(F_2): G_5].
\]
Furthermore, $G\cap G_5$ is orientable since $G_5$ is. 
\end{proof}

\section{Pinching and filling tori}\label{S - 1.5}

  The spaces in  \ts\ref{S - 1} are obtained in two steps, by a procedure which can be described as ``pinching and systolic filling'' starting from a flat torus. 
  
Theorem \ref{T - rigid} shows that  every such a construction, using a family of flat tori, will have $X_0$ as a universal cover, provided it satisfies a few basic conditions, described in the following proposition.
  
\begin{proposition} Let $t\geq 1$ be an integer. Suppose that:
\begin{enumerate}
\item
  $T_1,\ldots T_t$ is a finite family of flat tori, endowed with a simplicial metric structure in which every cell is a lozenge with sides of length 1
\item
  $\sigma$ is a fixed point free involution on the vertex set of $T:= \bigsqcup_{k=1}^t T_k$
\item  the systolic length in $T':=T/\ip{\sigma}$ is 3
\item every edge in $T'$ belongs to a unique systole of length 3
\item the systolic filling $ B$ of $T'$, obtained by attaching isometrically an equilateral triangle to every systole in $T'$, is locally CAT(0) (i.e., the link girth in $B$ is $\geq 2\pi$)
  \end{enumerate}
then $\widetilde B\simeq X_0$.
\end{proposition}
 
 \begin{proof}
 By Theorem \ref{T - rigid} it is enough to prove that $B$ is of type $\Aut(F_2)$. Since $\sigma$ is fixed point free, the link at a vertex in $B$ contains a union of two disjoint circles of length $2\pi$. 
 
 We shall use the notation in Lemma \ref{L - link}.  Since every edge belongs to a unique systole of length 3, the systolic filling provides an involution $\tau$ of the set  $\{1,2,3,4\}\sqcup \{1',2',3',4'\}$ of vertices in the link. Since $B$ is locally CAT(0), and the edge length from the systoles are $\pi/3$, the involution $\tau$ induces a bijection from $\{1,2,3,4\}$ to $\{1',2',3',4'\}$.
 We may assume without loss of generality that $\tau(1)=2'$. By the CAT(0) condition, it follows that $\tau(4)\neq 1'$. 
 
 Suppose that $\tau(4)=4'$. Then the distance between $3$ and $3'$ is $\leq \pi$, which implies $\tau(2)=3'$ and $\tau(3)=1'$. In this case, however, the cycle $212'3'$ is of length $<2\pi$, which is a contradiction. Thus, $\tau(4)\neq 4'$. This implies that $\tau(4)=3'$.  Applying again the CAT(0) condition, we must have $\tau(2)=4'$ and $\tau(3)=1'$. 
 
 Labelling the angles of the faces as in \ts\ref{S - 2}, the above shows that the link of vertex in $B$ is label isomorphic to the link of type $\Aut(F_2)$, where the labels $t$ are associated with the systolic filling.   This implies that $B$ is of type $\Aut(F_2)$ and therefore that $X_0\simeq \tilde B_5$.
 \end{proof}

Furthermore, every (sufficiently deep) orientable torsion free subgroup of finite index is constructed in this way:

\begin{proposition}
Let $G$ be an orientable torsion free subgroup of finite index in $\Aut(F_2)$. Suppose that the injectivity radius of $X_0/G$ is  $>1$. Then there exists a finite family of tori $T_1,\ldots, T_t$ and a fixed point free involution $\sigma$ on the vertex set of $T:=\sqcup_{k=1}^t T_k$, such that $T':= T/\ip \sigma$ satisfies the condition in the previous proposition, and the systolic filling $B$ of $T'$ is isometric to $X_0/G$. 
\end{proposition}

\begin{proof}  We say that two edges $e$ and $f$ in $X_0$ (or $X_0/G$) are equivalent if there exists a gallery $(f_1,\ldots, f_n)$ containing them, such that $f_i$ is a losenge for every $i$. 

Let $e$ be an edge in $X_0$ and $\tilde e$ be a lift of $e$ in $X_0$. It is clear that the equivalence class $[\tilde e]$ of $\tilde e$ maps surjectively onto the equivalence class of $[e]$ under the covering map $\pi\colon X_0\surj X_0/G$. The convex hull $H$ of $[\tilde e]$ is isometric to a flat plane tessellated by lozenges. We let $T_e'$ denote the image of $H$ under $\pi$.

Say that a vertex $x\in T_e'$ is a double point if the link of $T_e'$ at $x$ is a disjoint union of circles. The map $H\to T_e'$ factorize through a map $H\to T_e\to T_e'$, where $T_e$ is obtained from $T_e'$ by blowing up every double point.  Since $X_0/G$ is compact, so is $T_e'$.  Therefore, $T_e$ is compact. Since $H\surj T_e$ is a covering map and $G$ is orientable, it follows that $T_e$ is a torus. We let $\sigma_e$ be the partially defined involution on $T_e$ inducing the quotient map $T_e\surj T_e'$. 

Let $e_1,\ldots, e_t$ be a representative set of equivalence classes of edges in $X_0/G$. Associated with the $e_k$'s are tori $T_k$ and partially defined involution $\sigma_k$ on $T_k$ such that the edge set of $T_k/\ip \sigma_k\subset X_0/G$ coincides with the equivalence class of $e_k$. 

Furthermore, for every vertex $x\in T_k$ not in the domain of $\sigma_k$, there exists a unique $k'\neq k$ such that $x$ is a vertex of $T_k'$. This defines an involution $\sigma_0$ on the complement of $\sqcup_{k} \sigma$ to itself on $T:=\sqcup_{k} T_k$.  This involution is a fixed point free involution on $T$. Since the injectivity radius of $X_0/G$ is $>1$, the systolic length of $T/\ip \sigma$  is $\geq 3$, and therefore $X_0/G$  is the systolic completion of $T/\ip \sigma$ in the sense of the previous proposition.  
\end{proof}

The map $T\to B$ can be viewed as a structure of ``space with jumps'' on the torus (or union of tori) $T$. The geodesics with respect to such a structure in $T$ are allowed to jump between certain transverse codimension 1 subspaces they cross (in the present situation, it is the 1-skeleton, which are the sides of the triangles). The length of the jump, and the incidence angles are described by the geometry of the added triangles. A pinching occurs along a codimension 2 subspaces (intersection of codimension 1 subspaces), which are singular sets, corresponding to instantaneous jumps of a geodesic between two points in $T$.  We will not attempt to formalize this notion further in the present paper. 

\begin{remark} 
The number $t$ of tori, and the geometric parameters of the individual tori, provide conjugacy invariants for the given subgroup $G$. As mentioned in the introduction, the description of the family of subgroups with a prescribed invariant, e.g., the torsion free finite index subgroups $G$ of $\Aut(F_2)$ with  a given torus number $t(G)$, seems rather involved however.
\end{remark}
\section{A group cobordism for $\Aut(F_2)$}\label{S - 3}

In this section we show that the surgery techniques from \cite{surg}  which were used to construct (in many cases, infinitely many) groups of a given  type, can be applied to the group $\Aut(F_2)$. (Indeed, this is how the toric presentation in \ts \ref{S - 1} and the groups $G_n$ were found.)

Let $A$ be a (e.g., labelled) type. A category $\Bord_A$ of group cobordisms of type $A$ can be defined as follows.  The objects in this category are called collars, and the arrows, group cobordisms; in the present paper we only discuss the case where $A$ is the type $\Aut(F_2)$ defined in \ts \ref{S - 2}. 

Let us first review the notion of collar.
 An (abstract) \emph{open collar} is a topological space of the form $H\times (0,1)$ where $H$ is a graph (not necessarily connected).
If $X$ is a 2-complex, an open collar in $X$ is, by definition, an embedding $C\colon H\times (0,1)\inj X$. We shall refer to the domain $H\times (0,1)$ as the abstract collar defining  $C$. The \emph{dual} of an open collar of $X$ is the open collar $C'\colon H\times (0,1)\inj X$ defined by $C'(x,t):=C(x,1-t)$. The \emph{collar closure} of $C$  the topological closure $\overline C$ of the image of $C$ in $X$;
the \emph{span} of $C$ in  $X$ is the set $\spn(C)$  of vertices of $X$ contained in collar closure of $C$; the \emph{simplicial closure} of $C$ is  is the union of  all the open edges and open faces it intersects.  As in \cite{surg} we  only consider collars which are simplicially closed and
 vertex free.

We shall denote by  $\Bord_{\Aut(F_2)}$ the category of group cobordisms of type $\Aut(F_2)$.  
 We construct an object $C$ in $\Bord_{\Aut(F_2)}$ as follows.  

Fix an integer $y\in \IN$. We use the notation introduced at the end of \ts \ref{S - 1}. We will view $C$ as a ``slice'' of the cylinder $T_\infty$. We fix four letters $A_y,B_y,C_y,D_y$ respectively on  $((x,y),(x,y-1))$ where $x=0,2,4$ and on $((1,y),(2,y))$.
Recall that for every letter $L$ on $((x,y),(x,y-1))$, where $x$ is even, we write labels $L'$ and $L''$ on, respectively, $((x,y-2),(x+1,y-2))$ and   $((x+1,y-3),(x+1,y-4))$, while for a letter $L$ on $((1,y),(2,y))$, we write labels $L'$ and $L''$ on, respectively, $((3,y-2),(4,y-2))$ and   $((5,y-4),(6,y-4))$. 

By definition, the cylinder $T_\infty$ is  a quotient of a strip $[0,6]\times \IR$ using the twist 
$\tau_{-6}$ in the vertical direction. Recall that a gallery is a sequence of faces $(f_1,\ldots, f_n)$ such that $f_i\cap f_{i+1}$ is an edge. 

We say that a gallery in $T_\infty$ is \emph{generating} if it is closed (i.e., cyclic permutations remain galleries) and homotopic to an element generating  $\pi_1(T_\infty)$ . 

\begin{lemma} The minimal generating gallery has length $n=12$.   
\end{lemma}

\begin{proof}
Indeed, writing $T_\infty$ as a quotient of a strip $[0,6]\times \IR$ of size $6\times \infty$ by $\tau_{-6}$, the gallery distance between a boundary edge on $\{0\}\times \IR$  and its image by $\tau_{-6}$ in $\{6\}\times \IR$ is $12=6+6$.
\end{proof}

The collar $C$ will be built from a minimal generating gallery on $T_\infty$. Starting from the edge labelled $A_y$, the gallery is defined by the succession of edges $f_i\cap f_{i+1}$. The edges have the following labels: 
\[
A_y, A_{y+1}',A_y', A_{y+1}'',B_{y-2}, B_{y-1}',B_{y-2}', B_{y-1}'',C_{y-4}, C_{y-3}',C_{y-4}', C_{y-3}''
\]   
Note the corresponding gallery $(f_1,\ldots, f_{12})$ is closed: every change of letter occurs with a drop of $-2$ for a total drop of $-6$, which is consistent with $\tau_{-6}$. This defines a ``zig-zag'' gallery  generating $\pi_1(T_\infty)$. 

As a topological space the gallery $(f_1,\ldots, f_{12})$ is homeomorphic to $[0,1]\times \SI^1$. We shall refer to the gallery minus its boundary as  open.  

\begin{definition}
Let $C$ be the union of
\begin{enumerate}
\item the image of the open generating gallery $(f_1,\ldots, f_{12})$ in the basic construction $B_\infty$.
\item the triangles in  $B_\infty$ associated with the following six triples (knights) $K=(L,L',L'')$ on the  letters 
\[
L=A_{y+1}, A_{y}, B_{y-1},B_{y-2}, C_{y-3},C_{y-4}
\]
where every triangle associated with a triple $K$ is semi-open, in the sense that it does not contain the (unique) edge not belonging to the image of the gallery.
\end{enumerate}
\end{definition}

\begin{lemma}
$C$ is a product space. 
\end{lemma}

\begin{proof}
It is clear that the open gallery is a product space homeomorphic to $(0,1)\times \SI^1$. Under this identification, the added triples $K$ define a space of the form $(0,1)\times H$ where $H$ is a finite graph (the nerve) obtained by adding 6 edges to $\SI^1$. 
\end{proof}

One can of course give an explicit description of $H$:

\begin{lemma}
The graph $H$ is isomorphic to the Cayley graph of $\ZI/12\ZI$, with respect to 1, together with an additional edge $(n,n+2)$ for every $n\equiv 0,1 \mod 4$.
\end{lemma}  

Therefore, we may view $C$ as a open collar in $B_\infty$ under the identity mapping $C\to B_\infty$.

\begin{lemma}
$C$ is a full collar in $B_\infty$ 
\end{lemma}

\begin{proof}
Every open edge $e=f_i\cap f_{i+1}$ in $C$ belongs to a (unique) triple $K$, and therefore every point in $e$ has an open neighbourhood  included in $C$.  
\end{proof}

Since $B_\infty$ is a complex of type $\Aut(F_2)$, the above shows that the isomorphism class of $C$ is an object in the category $\Bord_{\Aut(F_2)}$.

The arrows in $\Bord_{\Aut(F_2)}$ are group cobordisms:

\begin{definition}\label{D - group cobordisms}
A \emph{group cobordism} is a 2-complex $B$ together with a pair $(C,D)$ of collars of $B$ whose boundaries $\del^-C$ and $\del^+ D$ form a partition of the topological boundary of $B$:
\[
\del B= \del^-C\sqcup \del ^+ D.
\] 
\end{definition} 

Let us  construct the group cobordism $B$ of type $\Aut(F_2)$. The collar $C$  depends on $y\in \IN$, however, it is clear that $C_y\simeq C_{y+1}$. The cobordism $B$ has $C$ as domain and codomain.

\begin{definition}\label{D - cobordism def}
Let $B$ be the union of 
\begin{enumerate}
\item $C_y\cup C_{y+1}$
\item the closed triangle in $B_\infty$ associated with the triple $K=(L,L',L'')$ on the  letters $L=D_{y-2}$.
\end{enumerate}
\end{definition}
Again, $B$ depends on $y$, where $B_y$ is isomorphic to $B_{y+1}$ and defines a unique arrow, again denoted $B$, in $\Bord_{\Aut(F_2)}$.  
The inclusion map $L_B,R_B\colon C\to B$ (left and right collar boundary) and the obvious inclusion of $C$ as $C_y$ and $C_{y+1}$. 

In particular:

\begin{theorem} The map taking 1 to $B$ induces a unital  inclusion 
$\IN\to \Bord_{\Aut(F_2)}$. 
\end{theorem}

\begin{proof}
Indeed, $B^{\circ n}\neq B^{\circ m}$ if $n\neq m$, where $B^{\circ n}$ refers to the $n$-fold composition $B\circ \cdots \circ B$ in $\Bord_{\Aut(F_2)}$.
\end{proof}  

In the language of \cite{surg}, the above shows the following:

\begin{theorem}
 $\Aut(F_2)$ is virtually accessible by surgery.
\end{theorem}

This means that $\Aut(F_2)$ admits a finite index subgroup  which is the fundamental group of a complex obtained by a surgery construction in a cobordism category (see \cite[\ts 10]{surg}).  Here the groups $G_n$ are of finite index in $\Aut(F_2)$ and the fundamental groups of the complexes $B_n$, which are of type $\Aut(F_2)$ defined by a surgery construction in $\Bord_{\Aut(F_2)}$.

We take this opportunity to make a correction to \cite[Lemma 17]{autf2puzzles}. At the bottom of the page it is stated  that ``there are two extensions of this section'': it should be ``three extensions''. Namely, in the first case (when the lozenges on the south-east triangles are oriented pointing south) one extension is the 3-strip, as indicated, which amounts to extending the lozenges with two triangles. A third sort of extension uses lozenges instead. In this case, the lozenges belong to a (using the terminology in \cite{autf2puzzles}) semi-infinite $\diamond$-strip of type $2\times \infty$. This can be visualized using the surgery construction above: starting from the closed triangle defined in $B$ above, Def.\ \ref{D - cobordism def}, (2), one may use three lozenges belonging to a single collar (either all belonging to $C_y$, or all in $C_{y+1}$) which can be extended into three semi-infinite $\diamond$-strip of type $2\times \infty$ in the universal cover (so the resulting puzzle has an order 3 symmetry). 

\section{Complements to Theorem \ref{T - rigid}} \label{S - 4}

We conclude some remarks on Theorem \ref{T - rigid}, regarding spaces locally isometric to $X_0$.  
It is an interesting exercise to construct groups acting freely uniformly on a CAT(0) 2-complex locally isometric to the 2-complex $X_0$ of $\Aut(F_2)$ (but not isometric to it), in the sense that their link are isometric to the link of $X_0$. In the present section, we provide one example. 

By Theorem \ref{T - rigid} such a complex $X'$ is not of type $\Aut(F_2)$.  The example will be of the following type. 

Let $A$  denote the metric type (i.e., a set of metric graphs, and a sets of shapes) defined by:  
\begin{enumerate}
\item Graph:  the link of the Brady complex with the angular metric (see \ts\ref{S - 2}). 
\item Shapes: an equilateral triangle, and an hexagon with sides of length 1. Both are viewed as standard polygons in the Euclidean plane with the induced metric.  
\end{enumerate}

By definition, every CAT(0) 2-complex of type $A$ is locally isometric to $X_0$ but not isometric to it.

\begin{proposition}
There exists a group $G'$ acting freely uniformly isometrically on a CAT(0) 2-complex $X'$ of type $A$.
\end{proposition}

The construction is as follows. 
We begin with a single hexagon on a set of 6 vertices, which we  denote 
\[
\{1^+,2^-,3^+,1^-,2^+,3^-\},
\] 
and edges labelled from 1 to 6 in a cyclic order as follows.  

\begin{figure}[H]
\includegraphics[width=6cm]{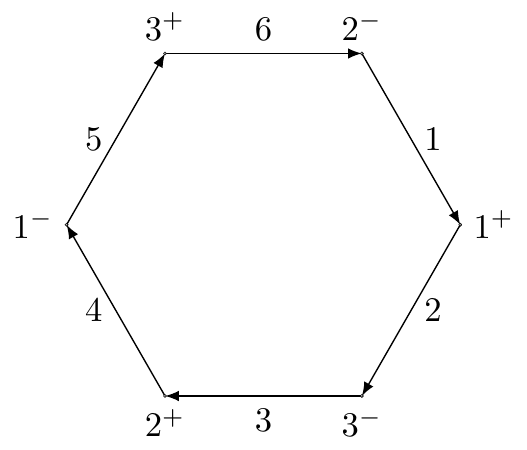}
\end{figure}

We shall realize these 6 vertices as the vertex set of a locally CAT(0) space of type $A$, containing the hexagon as a face. 

Consider additional edges between these vertices: 

\begin{enumerate}
\item[] two edges between $i^-$ and $i^+$
\item[] two edges between $i^+$ and $(i+1)^+$
\item[] two edges between $i^-$ and $(i+1)^-$
\end{enumerate}
(where  $i$ is an index modulo 3) organized and named as follows:

\begin{figure}[H]
\includegraphics[width=10cm]{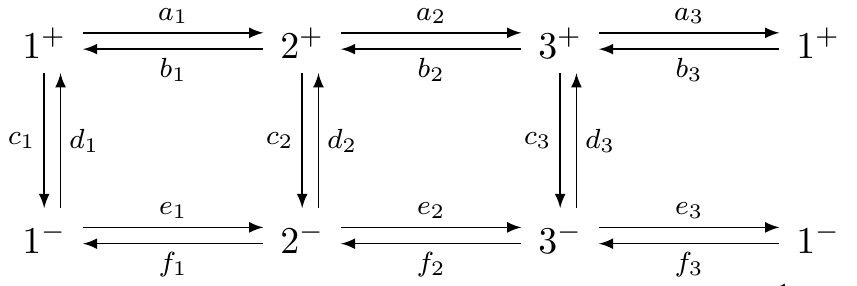}
\end{figure}

Together with the edges of the hexagon, this defines a regular graph of order 8. Note that this graph has a natural symmetry $\sigma$ of order 3 taking evert letter $l_i$ to the letter $l_{i+1}$ (modulo 3).

Consider the following hexagon and four triangles
\begin{enumerate}
\item[] $(a_1,c_2,f_1,e_1,d_2,b_1)$
\item[] $(d_1,a_1,4)$ \ \ $(f_1,d_1,1^-)$ 
\item[] $(b_1,c_1,4^-)$ \ \ $(c_1,e_1,1)$ 
\end{enumerate}
Together with their images under $\sigma$, this defines 3 triangles and 12 triangles. 
In addition to these triangle add the four triangles:
\begin{enumerate}
\item[] $(a_1,a_2,a_3)$\ \ $(b_1,b_2,b_3)$ 
\item[] $(e_1,e_2,e_3)$\ \ $(f_1,f_2,f_3)$
\end{enumerate}
This defines a 2-complex, whose fundamental group is $G'$ and universal cover $X'$. It is immediate to check that:

\begin{lemma}
The link of $X'$ is isometric to the link of $X_0$.
\end{lemma}
\begin{proof}
Note that it is enough to check a single vertex, since $\sigma$ and the reflection with respect to the horizontal axis extend to the 2-complex.

We may index the vertex set of the link by $a_1,b_1a_3,b_3,c_1,d_1,1,2$, where the latter two numbers are associated with the initial hexagon. There are four hexagon edges: $(a_1,b_1)$, $(a_3,c_1)$, $(b_3,d_1)$, and $(1,2)$ (for the first hexagon). One can then draw the edge associated with triangles: these are $(a_1,d_1)$, $(d_1,1)$, $(b_1,c_1)$, $(c_1,1)$, from the images under $\sigma$: $(b_3,2)$, $(a_3,2)$, and finally, $(a_1,a_3)$ and $(b_1,b_3)$. 

It is not difficult to show that this graph is isometric to the link of $X_0$.     
\end{proof}

 We also note that: 

\begin{proposition}
 $\Aut(X')$ is vertex transitive.
\end{proposition}

This is part of the argument in the previous lemma.

\end{document}